\documentclass[12pt]{amsart}

\usepackage{verbatim, amssymb, enumitem, mathtools,graphicx}

\usepackage{soul}

\usepackage[breaklinks=true,colorlinks=true,linkcolor=blue,citecolor=red,urlcolor=blue,psdextra,pdfencoding=auto]{hyperref}
\allowdisplaybreaks

\renewcommand \a{\alpha}
\renewcommand \b{\beta}

\newcommand \la{\lambda}
\newcommand \ve{\varepsilon}

\newcommand \br{\mathbb{R}}
\newcommand \bc{\mathbb{C}}
\newcommand \bh{\mathbb{H}}
\newcommand \bo{\mathbb{O}}

\newcommand \Ker{\operatorname{Ker}}

\renewcommand \Re{\operatorname{Re}}
\newcommand \Span{\operatorname{Span}}
\newcommand \Tr{\operatorname{Tr}}
\newcommand \db{\partial}
\newcommand \SO{\mathrm{SO}}
\newcommand \Sp{\mathrm{Sp}}
\newcommand \SU{\mathrm{SU}}

\newcommand \Ff{\mathrm{F}_4}
\newcommand \Spin{\mathrm{Spin}}

\newcommand \cI{\mathcal{I}}
\newcommand \cB{\mathcal{B}}

\newcommand \cL{\mathcal{L}}

\newcommand \cq{\mathrm{q}}

\newcommand \sK{\mathsf{K}}
\newcommand \sS{\mathsf{S}}
\newcommand \sL{\mathsf{L}}

\newcommand \soo{\mathfrak{so}}
\newcommand \spg{\mathfrak{sp}}

\newcommand \f{\mathfrak{f}}

\newcommand \ri{\mathrm{i}}

\newcommand \Sym{\operatorname{Sym}}

\newcommand \Ad{\operatorname{Ad}}

\newcommand \<{\langle}
\renewcommand \>{\rangle}

\newtheorem{theorem}{Theorem}
\newtheorem*{theorem*}{Theorem}

\newtheorem*{corollary*}{Corollary}
\newtheorem*{conj*}{Conjecture}
\newtheorem{lemma}{Lemma}

\newtheorem*{prop*}{Proposition}

\theoremstyle{definition}

\newtheorem*{definition*}{Definition}

\theoremstyle{remark}

\newtheorem*{notation*}{Notation}
\newtheorem*{algorithm*}{Algorithm}
\newtheorem*{example*}{Example}

\makeatletter
\@namedef{subjclassname@2020}{%
  \textup{2020} Mathematics Subject Classification}
\makeatother

\begin{document}

\title[Killing tensors on symmetric spaces]{Quadratic Killing tensors on symmetric spaces which are not generated by Killing vector fields}

\author{Vladimir   S.  Matveev}
\address{Fakult\"{a}t f\"{u}r Mathematik und Informatik, Friedrich-Schiller-Universit\"{a}t, 07737 Jena, Germany}
\email{vladimir.matveev@uni-jena.de}
\thanks{The first named author was partially supported by ARC Discovery Grant DP210100951 and by the  DFG   projects 455806247 and 529233771.  The first  named author is thankful to La Trobe University for hospitality.}

\author{Yuri Nikolayevsky}
\address{Department of Mathematical and Physical Sciences, La Trobe University, VIC 3086, Australia} 
\email{Y.Nikolayevsky@latrobe.edu.au}
\thanks{The second named author was partially supported by ARC Discovery Grant DP210100951. The second named author is thankful to Friedrich-Schiller-Universit\"{a}t for hospitality.} %



\begin{abstract}
Every Killing tensor field on the space of constant curvature and on the complex projective space can be decomposed into the sum of  symmetric tensor products of Killing vector fields (equivalently, every polynomial in  velocities integral of the geodesic flow is a polynomial in the linear integrals). This fact led to the natural question on whether this property is shared by Killing tensor fields on all Riemannian symmetric spaces. We answer this question in the negative by constructing explicit examples of quadratic Killing tensor fields which are not quadratic forms in the Killing vector fields on the quaternionic projective spaces $\bh P^n, \, n \ge 3$, and on the Cayley projective plane $\bo P^2$.
\end{abstract}

\maketitle

\section{Introduction}
\label{s:intro}

A \emph{Killing tensor field} $K=K(x)_{i_1\dots i_d}$ of rank $d \ge 1$ on a Riemannian manifold $(M,ds^2=g_{ij}x^ix^j)$ is a symmetric tensor field that satisfies the Killing equation
\begin{equation}\label{eq:defK}
  K_{(i_1\dots i_d,j)}=0,
\end{equation}
where the comma denotes the covariant derivative and the parentheses denote the symmetrisation by all indices. This definition is equivalent to the fact that the function $\xi\in T_xM \mapsto K(x)_{i_1\dots i_d} \xi^{i_1} \cdots \xi^{i_d}$ polynomial in the velocities is an integral of the geodesic flow of $(M,ds^2)$: for any naturally parameterised geodesic $s \mapsto \gamma(s)$ of $(M,ds^2)$, the function $s \mapsto K(\gamma(s))_{i_1\dots i_d} (\dot{\gamma}(s))^{i_1} \dots (\dot{\gamma}(s))^{i_d}$ is constant.

Killing tensor fields of rank $d=1$ are called \emph{Killing \emph{(}co\emph{)}vector fields}. It is well known that a vector field is Killing if and only if the $1$-parametric group of diffeomorphisms of $M$ which it generates is a group of isometries.  To the best of our knowledge, no such geometric definition is available for Killing tensor fields of degree $d \ge 2$, though of course,  Killing tensors  correspond to the so-called hidden symmetries of the geodesic flow, that is, to the Hamiltonian vector fields which commute with the generator of the geodesic flow.

Denote $\sK^d(M)$ the space of Killing tensor fields of rank $d$ on $(M,ds^2)$. For every $d \ge 1$, the space $\sK^d(M)$ is finite dimensional. Moreover, the space $\sK(M) = \br \oplus \sK^1(M) \oplus \sK^2(M) \oplus \dots$ of all Killing tensor fields on $(M,ds^2)$ is an associative,  commutative algebra relative to the symmetric tensor product. In particular, the polynomial algebra $\sS(M)$ generated by Killing vector fields is a subalgebra of $\sK(M)$ (it should be noted that the homomorphism from the polynomial algebra generated by $\sK^1(M)$ to $\sS(M)$ may have a nontrivial kernel). We will call the elements of $\sS(M)$ \emph{decomposable}, and the elements of $\sK(M) \setminus \sS(M)$ \emph{indecomposable}. The fact that not every Killing tensor field of rank $d \ge 2$ is decomposable, even if one disregards the polynomials in the metric tensor, is well known (since at least Darboux): for example, a $2$-dimensional Liouville metric $ds^2=(\la(x)+\mu(y)) (dx^2 + dy^2)$ has quadratic Killing tensors not proportional to the metric, but in general, a trivial isometry group. On the other hand, any Killing tensor field on the space of constant curvature is decomposable~\cite{Tho,ST1, MMS}. This fact suggested the following question \cite[Question~3.9]{BMMT}: ``\emph{In a symmetric space, is every Killing tensor field decomposable?}" The answer to this question is known to be in the positive for the complex projective space~\cite[Corollary~5]{East}, \cite[Theorem~2.2]{ST2}, and one may expect that because symmetric spaces (especially of rank one) have such a large isometry group, any Killing tensor field on them must be decomposable.

Surprisingly, this is not true already for quadratic Killing tensor fields on all but one ``remaining" rank-one symmetric spaces, namely on the quaternionic projective spaces $\bh P^n=\Sp(n+1)/(\Sp(n)\Sp(1))$ with $n \ge 3$, and on the Cayley projective plane $\bo P^2 = \Ff/\Spin(9)$. We prove the following.

\begin{theorem} \label{t:hpn}
  For $n \ge 3$, the space $\sK^2(\bh P^n)$ contains a subspace \emph{(}in fact, an irreducible $\Sp(n+1)$-module\emph{)} of dimension $\frac16(n - 2)(n + 1)(2n + 1)(2n + 3)$ all of whose nonzero elements are indecomposable.
\end{theorem}

\begin{theorem} \label{t:op2}
  The space $\sK^2(\bo P^2)$ contains a subspace \emph{(}in fact, an irreducible $\Ff$-module\emph{)} of dimension $26$ all of whose nonzero elements are indecomposable.
\end{theorem}

In the paper which is currently in preparation we will show that the quadratic Killing tensors on $\bh P^2$ are decomposable, and that $\sK^2(\bh P^n)$ with $n \ge 3$ (respectively, $\sK^2(\bo P^2)$) is the direct sum of the module of decomposable quadratic Killing tensors and the module from Theorem~\ref{t:hpn} (respectively, from Theorem~\ref{t:op2}). We will also establish a connection between the algebras of Killing tensors on dual irreducible symmetric spaces; this will enable one to extend the above results to noncompact rank-one symmetric spaces.

The proofs of both theorems follow the same scheme: we first construct the required submodule explicitly, and then check that the decomposition of the module of quadratic decomposable Killing tensor fields (which is isomorphic to the $G$-module $\Sym^2(\Ad_G)$ for $G=\Sp(n+1)$ and $G = \Ff$) does not contain that submodule.

\section{Quaternionic projective space. Proof of Theorem~\ref{t:hpn}}
\label{s:hpn}

We equip the space $\br^{4n+4}$ with the quaternionic structure defined by three anticommuting Hermitian structures $J_1, J_2$ and $J_3=J_1J_2$. 

We denote $\spg(1)=\Span(J_1,J_2,J_3)$ and define $\spg(n+1) \subset \soo(4n+4)$ to be the centraliser of $\spg(1)$ in $\soo(4n+4)$. Furthermore, let $V_{n+1}$ be the space of symmetric operators in $\br^{4n+4}$ which commute with $\spg(1)$. Then $V_{n+1}$ is an $\spg(n+1)$-module, which can be decomposed as $V_{n+1} = \br I_{4n+4} \oplus W_{n+1}$, where $W_{n+1}$ is the subspace of $V_{n+1}$ consisting of operators with trace zero. The $\spg(n+1)$-module $\br I_{4n+4}$ is trivial, and the $\spg(n+1)$-module $W_{n+1}$ (of dimension $n(2n+3)$) is irreducible (it can be viewed as the tangent space of the irreducible symmetric space $\SU(2n+2)/\Sp(n+1)$ at the projection of the identity).

As $V_{n+1}$ commutes with $\spg(1)$, all the operators $SJ_\a,\; S \in V_{n+1}, \a=1,2,3$, are skew-symmetric. For $S \in V_{n+1}$, we define the constant tensor $T_S$ of type $(0,4)$ on $\br^{4n+4}$ by
\begin{equation} \label{eq:hTS}
    T_S(X,Y,P,Q)=\frac12 \sum_{\a=1}^{3} (\<SJ_\a X, P\>\<SJ_\a Y, Q\>+\<SJ_\a X, Q\>\<SJ_\a Y, P\>),
\end{equation}
where $X,Y,P,Q \in \br^{4n+4}$, so that $T_S(X,X,P,P)=\sum_{\a=1}^{3} \<SJ_\a X, P\>^2$. It is easy to see that the tensor $T_S$ is symmetric in the first two arguments and in the second two arguments, and satisfies the first Bianchi identity by any three arguments (this follows from the fact that $SJ_\a$ are skew-symmetric).

 Denote $\pi: S^{4n+3} \to \bh P^n$ the natural projection (which is a Riemannian submersion) from the unit sphere $S^{4n+3} \subset \br^{4n+4}$ along the fibers of the Hopf fibration defined by the orbits of the action of $\Sp(1)$ on $S^{4n+3}$.

We claim that for any $S \in V_{n+1}$, the tensor $T_S$ defined by~\eqref{eq:hTS} is $\Sp(1)$-invariant, and hence descends under $\pi$ to a quadratic tensor field on $\bh P^n$, that is, for $X \in S^{4n+3}$ and $P \in T_X S^{4n+3}$, the quadratic tensor field $F_S$ on $\bh P^n$ defined by
\begin{equation} \label{eq:FS}
F_S(\pi(X),\pi(X),(d\pi)_X P,(d\pi)_X P)=T_S(X,X,P,P)
\end{equation}
is well-defined. To see this, given $S \in V_{n+1}$ and $\beta=1,2,3$, we compute the Lie derivative $\cL_{v_\b}T_S$ of the tensor field $T_S$ along the vector field $v_\b(X) = J_\b X$ on $\br^{4n+4}$. We obtain $(\cL_{v_\b}T_S)(X,X,P,P)= 2 \sum_{\a=1}^{3} \<[SJ_\a,J_\b] X, P\>\<SJ_\a X, P\> = 2 \sum_{\a=1}^{3} \<S[J_\a,J_\b] X, P\>$ $\<SJ_\a X, P\>$, as $S$ commutes with $J_\b$, and so $(\cL_{v_\b}T_S)(X,X,P,P)=0$, as for $\a \ne \b$ we have $[J_\a,J_\b]=2\ve J_\gamma$, where $\ve$ is the sign of the permutation $(\a,\b,\gamma) \mapsto (1,2,3)$. It follows that $\cL_{v_\b}T_S=0$, for all $\b=1,2,3$ and all $S \in V_{n+1}$, which proves our claim.

Furthermore, the quadratic tensor field $F_S$ defined by~\eqref{eq:FS} is Killing. To prove this, we note that the geodesics of $\bh P^n$ are the projections of the geodesics of $S^{4n+3}$ under $\pi$. For a geodesic $\gamma(s) = (\cos s) X + (\sin s) P$ parameterised by the arc length $s$, where $X, P \in \br^{4n+4}$, with $\|X\|=\|P\|= 1, \, X \perp P$, we have $T_S(\gamma(s),\gamma(s),\dot{\gamma}(s), \dot{\gamma}(s)) =\sum_{\a=1}^{3} \<SJ_\a X, P\>^2$, which does not depend on $s$.

The map $F:V_{n+1} \to \sK^2(\bh P^n),\; S \mapsto F_S$, defined by~\eqref{eq:FS} extends to the linear map $\Phi: \Sym^2(V_{n+1}) \to \sK^2(\bh P^n)$ from the space of quadratic forms on $V_{n+1}$ to $\sK^2(\bh P^n)$ (where we identify the symmetric square $\Sym^2(V_{n+1})$ with its dual using the natural inner product in the space of tensors). Relative to a basis $\cB=\{S_A\}, \, A=1, \dots, n(2n+3)+1$, for $V_{n+1}$, the map $\Phi$ is given by
  \begin{equation}\label{eq:PhiSym}
    (\Phi(Q))(\pi(X),\pi(X),(d\pi)_X P,(d\pi)_X P)= \sum_{A,B} Q_{AB} \Big(\sum_{\a=1}^{3} \<S_A J_\a X, P\>\<S_B J_\a X, P\>\Big),
  \end{equation}
  where $Q \in \Sym^2(V_{n+1})$ has components $Q_{AB}=Q_{BA}$ relative to the basis $\cB$ for $V_{n+1}$, and $X, P \in \br^{4n+4}$.

It is easy to see that $\Phi$ is a homomorphism of $\Sp(n+1)$-modules. However, it is not injective. Define the linear form $\Tr: V_{n+1} \to \br$ by $S \mapsto \Tr S$ for $S \in V_{n+1}$. We prove the following.
{
\begin{lemma} \label{l:Phiinj}
  Suppose $n \ge 3$. Let $Q \in \Sym^2(V_{n+1})$. Then $Q \in \Ker \Phi$ if and only if $Q$ is divisible by $\Tr$.
\end{lemma}
\begin{proof}
    In the quaternionic language, the module $V_{n+1}$ is the space of quaternionic Hermitian matrices acting on the module $\bh^{n+1}$ from the left (where $\Span(J_1, J_2, J_3)$ acts on $\bh^{n+1}$ by the right componentwise multiplication by imaginary quaternions). However, for the purposes of the proof it will be more convenient to work in the real settings.

    Choose a basis for $V_{n+1}$ as follows. In $\bh=\br^4$, with the basis $\{1, \ri, \mathrm{j}, \mathrm{k}\}$, we denote $L(u)$ and $R(u)$ the matrices of the left and of the right multiplication by the quaternion $u \in \bh$, respectively. We view $\br^{4n+4}$ as the direct, orthogonal sum $\oplus_{s=1}^{n+1} \br^4_s$, which gives the decomposition of the space of $4(n+1) \times 4(n+1)$ matrices into the corresponding $4 \times 4$ blocks. We choose a basis for $\br^{4n+4}$ by identifying each $\br^4_s$ with $\bh$ and choosing a basis $\{1, \ri, \mathrm{j}, \mathrm{k}\}$ in it. Relative to this basis, we can view the matrix of each of the operators $J_1, J_2$ and $J_3$ as the block-diagonal matrix with each $4 \times 4$ diagonal block being the matrix $R(\ri), R(\mathrm{j})$ and $R(\mathrm{k})$, respectively. Now for $1 \le i \le n+1$, we define $S_i$ to be the $4(n+1) \times 4(n+1)$ matrix having the identity matrix $I_4$ in the $i$-th diagonal $4 \times 4$ block, and zeros elsewhere. For $1 \le i < j \le n+1$, we define $S_{ij}(u),\, u \in \cq:=\{1, \ri, \mathrm{j}, \mathrm{k}\}$, to be the matrix containing $L(u)$ in the $(i,j)$-th $4 \times 4$ block, $L(u^*)$ in the $(j,i)$-th $4 \times 4$ block (where $u^*$ is the quaternion conjugate to $u$) and zeros elsewhere. Then the set of matrices $\cB=\{S_i,\, 1 \le i \le n+1\} \cup \{S_{ij}(u), \, 1 \le i < j \le n+1,\, u \in \cq\}$ is a basis for $V_{n+1}$. We will need two properties of this basis. First, for any subset $\cI \subset \{1, \dots, n+1\}$, of cardinality $\# \cI = m,\; 1 \le m \le n+1$, the subset of $\cB$ consisting of the matrices $S_k, S_{ij}(u)$, with $k, i, j \in \cI,\, i<j,\, u \in \cq$, forms a basis for the space $V_m$ constructed as above for the space $\bh^m(\cI) = \oplus_{s \in \cI} \br^4_s$. Second, for any two elements of $\cB$, there exists a subset  $\cI \subset \{1, \dots, n+1\}$, of cardinality $m \le 4$ such that the orthogonal complement to $\bh^m(\cI) = \oplus_{s \in \cI} \br^4_s$ lies in the common kernel of these two elements.

    For an element $S \in V_{n+1}$ with the decomposition $S=\sum_{p=1}^{n+1} a_p S_p + \sum_{i<j, u \in \cq} a_{i,j,u} S_{ij}(u)$ relative to the basis $\cB$ we have $\Tr S= \sum_{p=1}^{n+1} a_p$, and so for a quadratic form $Q$ on $V_{n+1}$ with the decomposition $Q=\sum_{p,q=1}^{n+1} b_{pq} S_p \odot S_q + \sum_{i<j, u \in \cq, p} c_{p,i,j,u} S_{ij}(u) \odot S_p + \sum_{i<j, k<l, u,v \in \cq} \mu_{i,j,u,k,l,v} S_{ij}(u) \odot S_{kl}(v)$ (where $b_{pq}=b_{qp}$ and $\mu_{i,j,u,k,l,v}= \mu_{k,l,v,i,j,u}$) the condition of the lemma saying that $Q$ is divisible by $\Tr$ means that for all $u,v \in \cq$ and all $i,j,k,l,p,q,r,s \in \{1, \dots, n+1\}$ with $i<j$ and $k<l$ we have
    \begin{equation}\label{eq:coeffs}
      \mu_{i,j,u,k,l,v} = 0,\quad c_{p,i,j,u}=c_{q,i,j,u}, \quad b_{pq}+b_{rs}=b_{ps}+b_{rq}.
    \end{equation}
    The condition that $Q$ belongs to the kernel of $\Phi$ means that the right-hand side of~\eqref{eq:PhiSym} is zero for all $P$ in the horizontal distribution of the Hopf fibration, that is, for all $X, P \in \br^{4n+4}$ such that $P \perp X, J_1X, J_2X, J_3X$.

    Now for the ``if" direction of the lemma, we assume that $Q$ is divisible by $\Tr$, and so $Q=(\sum_{p=1}^{n+1} S_p) \odot S'$, for some $S' \in V_{n+1}$. Then the right-hand side of~\eqref{eq:PhiSym} equals $\sum_{p=1}^{n+1} \sum_{\a=1}^{3} \<S_p J_\a X, P\>\<S' J_\a X, P\>= \sum_{\a=1}^{3} \<J_\a X, P\>\<S' J_\a X, P\>$, which is zero for all $X, P \in \br^{4n+4}$ with $P \perp X, J_1X,J_2X, J_3X$.

    For the ``only if" direction we first consider the case $n=3$. Then $\dim V_{n+1}=28, \, \dim \Sym^2(V_{n+1})=406$, and the fact that the right-hand side of~\eqref{eq:PhiSym} is zero for $P \perp X, J_1X,J_2X, J_3X$ gives a system of linear equations for the coefficients of $Q$. By a straightforward calculation, with the aid of computer algebra, one shows that it is equivalent to the system~\eqref{eq:coeffs}. Now assume $n \ge 4$. The fact that the right-hand side of~\eqref{eq:PhiSym} equals zero for all $P \perp X, J_1X,J_2X, J_3X$, implies that the same is true under an additional assumption that $X$ and $P$ are chosen from a subspace $\bh^m(\cI) = \oplus_{s \in \cI} \br^4_s$ for some $\cI \subset \{1,\dots, n+1\}$, with $\# \cI =m$. But for any quadruple of indices $(i,j,k,l), (i,j,p,q), (p,q,r,s)$ we can choose such a subset $\cI$ of cardinality $m = 4$ which contains that quadruple, and so the equations~\eqref{eq:coeffs} are satisfied by induction by $n \ge 3$.
  \end{proof}
  }

We have $V_{n+1}= \br I_{4n+4} \oplus W_{n+1}$, where $W_{n+1} = \Ker \Tr$. This gives the direct decomposition $\Sym^2(V_{n+1}) = \br (I_{4n+4} \odot I_{4n+4}) \oplus (I_{4n+4} \otimes W_{n+1}) \oplus (\Sym^2(W_{n+1}))$ into $\Sp(n+1)$-modules. Lemma~\ref{l:Phiinj} tell us that the restriction $\Psi$ of the homomorphism $\Phi: \Sym^2(V_{n+1}) \to \sK^2(\bh P^n)$ to $\Sym^2(W_{n+1})$ is injective.

For the complexification $\spg(n+1, \bc)$, the adjoint (complex) representation $\Ad$ is isomorphic to $R(2\pi_1)$, and so we have the decomposition $\Sym^2(\Ad)= R(4\pi_1) + R(2\pi_2) + R(\pi_2) + 1$ into irreducible $\spg(n+1, \bc)$-modules (in the notation of \cite{VO}), where $1$ is the $1$-dimensional trivial module. The complexified module $W_{n+1}$ is isomorphic to the irreducible representation $R(\pi_2)$, and we have the decomposition $\Sym^2(R(\pi_2)) = R(\pi_4) + R(2\pi_2)+R(\pi_2)+1$ for $n \ge 3$. As all the modules in these decompositions are of the orthogonal type (see e.g.~\cite[Theorem~11.3(c)]{Sam}), we have the same decompositions into irreducibles for the real representations of $\spg(n+1)$. But the module of decomposable quadratic Killing fields on $\bh P^n$ is $\Sym^2(\Ad)$, which does not contain the submodule $\Psi(R(\pi_4)) \subset \sK^2(\bh P^n)$ of the module $\Psi(\Sym^2(W_{n+1})) \subset \sK^2(\bh P^n)$. As $\Psi$ is injective, we obtain the module isomorphic to $R(\pi_4)$ (of dimension $\frac16(n - 2)(n + 1)(2n + 1)(2n + 3)$) of quadratic Killing fields on $\bh P^n$ all nonzero elements of which are indecomposable.

\section{Cayley projective plane. Proof of Theorem~\ref{t:op2}}
\label{s:op2}

Out of several equivalent descriptions of $\bo P^2$ we will work with the following one \cite{Freu,Har,Baez}. Denote $\bo$ the algebra of octonions with the conjugation ${}^*$ and the inner product $\<x,y\>=\frac12 (xy^*+yx^*)$ for $x,y \in \bo$.

Let $H_3(\bo)$ be the Albert algebra, the Jordan algebra of $3 \times 3$ Hermitian octonion matrices. Elements of $H_3(\bo)$ have the form
  \begin{equation} \label{eq:H3O}
    A=\left(
        \begin{array}{ccc}
          r_1 & x_3^* & x_2^* \\
          x_3 & r_2 & x_1 \\
          x_2 & x_1^* & r_3 \\
        \end{array}
      \right), \qquad x_1, x_2,x_3 \in \bo, \; r_1,r_2,r_3 \in \br,
  \end{equation}
with the Jordan multiplication given by $A \circ B = \frac12(AB+BA)$. The automorphism group of $H_3(\bo)$ is the exceptional Lie group $\Ff$. The action of $\Ff$ preserves the trace given by $\Tr(A)=r_1+r_2+r_3$, the squared norm given by $\|A\|^2=\Tr(A^2)$, and the determinant given by $\det(A)=r_1r_2r_3 + 2\Re(x_1x_2x_3) - r_1 \|x_1\|^2 - r_2 \|x_2\|^2 - r_3\|x_3\|^2$. The squared norm defines a positive definite inner product on $H_3(\bo)$.

The determinant defines a symmetric trilinear form $\Phi$ on $H_3(\bo)$ given by $\Phi(A_1,A_2,A_3)$ $=\frac16 \frac{\db^3}{\db t_1 \db t_2 \db t_3} \det(t_1 A_1 + t_2 A_2 + t_3 A_3)$, for $A_1,A_2,A_3 \in H_3(\bo)$ and $t_1,t_2,t_3 \in \br$. 

The Cayley projective plane $\bo P^2$ is the submanifold of $H_3(\bo)$ defined as follows:
  \begin{equation*}
    \bo P^2 = \{X \in H_3(\bo) \, | \, \Tr(X)=1,\, \Phi(A,X,X)=0, \; \text{for all } A \in H_3(\bo)\}.
  \end{equation*}
With the induced metric, $\bo P^2$ becomes a rank-one Riemannian symmetric space. The group $\Ff$ acts transitively on $\bo P^2$, so that $\bo P^2$ is a single orbit, which is the orbit of the element
  \begin{equation*}
        E=\left(
        \begin{array}{ccc}
          1 & 0 & 0 \\
          0 & 0 & 0 \\
          0 & 0 & 0 \\
        \end{array}
      \right) \in H_3(\bo).
  \end{equation*}
Moreover, any element of $\bo P^2$ can be represented as $(x,y,z)^t (x^*,y^*,z^*)$, where the octonions $x,y,z$ satisfy $\|x\|^2+\|y\|^2+\|z\|^2=1$ and associate, that is, $x(yz)=(xy)z$ \cite[Lemma~14.90, Theorem~14.99]{Har}. In particular, the tangent space to $\bo P^2$ at $E$ is given by
  \begin{equation} \label{eq:TEOP2}
        T_E\bo P^2=\{T(y,z)\, | \, y, z \in \bo\}, \quad \text{where } T(y,z)=\left(
        \begin{array}{ccc}
          0 & y^* & z^* \\
          y & 0 & 0 \\
          z & 0 & 0 \\
        \end{array}
      \right)
  \end{equation}
The group $\Ff$ acts transitively on the set of geodesics of $\bo P^2$, and any geodesic can be obtained by the action of an element of $\Ff$ from the geodesic
  \begin{equation} \label{eq:geod}
        \gamma(s)=\left(
        \begin{array}{ccc}
          (\cos^2 s)1 & (\cos s \sin s)1 & 0 \\
          (\cos s \sin s)1 & (\sin^2 s)1 & 0 \\
          0 & 0 & 0 \\
        \end{array}
      \right),
  \end{equation}
where $s \in \br$ is an arc length parameter.

Denote $V$ the hyperplane $\Tr A= 0$ in $H_3(\bo)$. As $\Ff$ preserves the trace, the hyperplane $V$ is a $26$-dimensional $\Ff$-module, which is the standard (the smallest nontrivial) real irreducible representation of $\Ff$.

For any $A \in V$, we define the quadratic tensor field $K_A$ on $\bo P^2$ as follows: for $X \in \bo P^2$ and $Y,Z \in T_X \bo P^2$, set $K_A(Y,Z)=\Phi(Y,Z,A)$. The mapping $V \ni A \mapsto K_A$ defines the structure of an $\Ff$-module on the space $\sL=\{K_A \, | \, A \in V\}$. By Schur's Lemma, this mapping is either an isomorphism, or $\sL=0$. But the latter is not the case: taking $Y=T(1,0), Z =T(0,1) \in T_E \bo P^2$ in the notation of~\eqref{eq:TEOP2} and $A \in V$ whose only nonzero elements are $A_{23}=A_{32}=1$ one easily sees that $K_A(Y,Z)=\frac13$. Therefore $\sL$ is an irreducible $26$-dimensional $\Ff$-module.

We claim that the quadratic tensor fields $K_A$ on $\bo P^2$ are Killing. We need to show that any $K_A$ takes constant values along any geodesic. It is sufficient to show that for any $A \in V$, the function $s \mapsto K_A(\dot{\gamma}(s),\dot{\gamma}(s))$ is constant, where $\gamma(s)$ is given by~\eqref{eq:geod}. Taking $A$ as in~\eqref{eq:H3O} with $\Tr A = 0$, a direct calculation shows that $K_A(\dot{\gamma}(s),\dot{\gamma}(s))= \frac16 \frac{\db^3}{\db t_1 \db t_2 \db t_3} \det((t_1+t_2)\dot{\gamma}(s)+t_3A)=-\frac13 r_3$.

We therefore obtain a $26$-dimensional (irreducible) $\Ff$-submodule $\sL \subset \sK(\bo P^2)$. To show that all its nonzero elements are indecomposable, we note that all the irreducible complex representations of the complexification of the algebra $\f_4$ are of the orthogonal type~\cite[Theorem~3.11(f)]{Sam}, and so the decomposition of the $\Ff$-module of decomposable quadratic Killing tensor fields into irreducible submodules is given by $\Sym^2(\Ad) = R(2\pi_4)+R(2\pi_1)+1$ (in the notation of~\cite{VO}), while $\sL$ is isomorphic to $R(\pi_1)$.


\begin{thebibliography}{BMMT}

\bibitem[Baez]{Baez}
J.~Baez, \emph{The Octonions}, Bull. Amer. Math. Soc. \textbf{39} (2002), 145-205.

\bibitem[BMMT]{BMMT}
A.~Bolsinov, E.~Miranda, V.~Matveev, S.~Tabachnikov, \emph{Open problems, questions, and challenges in finite-dimensional integrable systems}, Philos. Trans. Roy. Soc. A \textbf{376} (2018), 20170430

\bibitem[East]{East}
M.~Eastwood, \emph{Killing tensors on complex projective space}, arXiv: \url{https://arxiv.org/abs/2309.00589}.


\bibitem[Har]{Har}
F.~R.~Harvey, \emph{Spinors and Calibrations}. Academic Press, Boston, 1990.

\bibitem[Hel]{Hel}
S.~Helgason, \emph{Differential geometry, Lie groups, and symmetric spaces}. Pure and Applied Mathematics, 80. Academic Press, Inc. New York, London, 1978.

\bibitem[Freu]{Freu}
H.~Freudenthal, \emph{Zur ebenen Oktavengeometrie}, Indag. Math. \textbf{15} (1953), 195-200.

\bibitem[MMS]{MMS}   
R.~McLenaghan, R.~Milson, R.~Smirnov, \emph{ Killing tensors as irreducible representations of the general linear group}, C. R. Acad. Sci. Paris, Ser. I \textbf{339} (2004), 621-624.

\bibitem[Sam]{Sam}
H.~Samelson, \emph{Notes on Lie algebras}. Universitext, Springer-Verlag, New York, 1990.

\bibitem[ST1]{ST1}
T.~Sumitomo, K.~Tandai, \emph{Killing tensor fields on the standard sphere and spectra of $\SO(n+1)/(\SO(n-1) \times \SO(2))$ and $\mathrm{O}(n+1)/(\mathrm{O}(n-1) \times \mathrm{O}(2))$}, Osaka Math. J. \textbf{20} (1983), 51-78.

\bibitem[ST2]{ST2}
T.~Sumitomo, K.~Tandai, \emph{On the centralizer of the Laplacian of $P_n(\bc)$ and the spectrum of complex Grassmann manifold $G_{2,n-1}(\bc)$},
Osaka J. Math. \textbf{22} (1985), 123-155.

\bibitem[Tho]{Tho}
G.~Thompson, \emph{Killing tensors in spaces of constant curvature}, J. Math. Phys. \textbf{27} (1986), 2693-2699.

\bibitem[VO]{VO}
\`{E}.~B.~Vinberg, A.~L.~Onishchik, \emph{A seminar on Lie groups and algebraic groups}, Nauka, Moscow, 1988.

\end{thebibliography}
\end{document}